\documentclass[letterpaper, 10 pt, conference]{ieeeconf}  

\IEEEoverridecommandlockouts                              

\usepackage[noadjust]{cite}

\usepackage{amsmath} 
\usepackage{amssymb}  
\usepackage{algorithm}
\usepackage{algpseudocode}
\usepackage{graphicx}
\usepackage[caption=false,lofdepth,lotdepth]{subfig}

\DeclareMathOperator*{\elim}{e-lim}

\newtheorem{lemma}{Lemma}
\newtheorem{theorem}{Theorem}

\newtheorem{proposition}{Proposition}

\newtheorem{corollary}{Corollary}
\newtheorem{example}{Example}

\DeclareMathOperator{\interior}{int}

\title{\LARGE \bf
Duality Approach to Bilevel Programs with a Convex Lower Level
}

\author{Aur\'{e}lien Ouattara and Anil Aswani
\thanks{This work was supported in part by NSF Award CMMI-1450963 and the Philippine-California Advanced Research Institutes (PCARI).}
\thanks{Aur\'{e}lien Ouattara and Anil Aswani are with the Department of Industrial Engineering and Operations Research, University of California, Berkeley, CA 94720, USA 
        {\tt\small aurelien.ouattara@berkeley.edu, aaswani@berkeley.edu}}%
}

\begin{document}

\maketitle
\thispagestyle{empty}
\pagestyle{empty}

\begin{abstract}
Bilevel programs are optimization problems where some variables are solutions to optimization problems themselves, and they arise in a variety of control applications, including: control of vehicle traffic networks, inverse reinforcement learning and inverse optimization, and robust control for human-automation systems. This paper develops a duality-based approach to solving bilevel programs where the lower level problem is convex.  Our approach is to use partial dualization to construct a new dual function that is differentiable, unlike the Lagrangian dual that is only directionally differentiable.  We use our dual to define a duality-based reformulation of bilevel programs, prove equivalence of our reformulation with the original bilevel program, and then introduce regularization to ensure constraint qualification holds.  These technical results about our new dual and regularized duality-based reformulation are used to provide theoretical justification for an algorithm we construct for solving bilevel programs with a convex lower level, and we conclude by demonstrating the efficacy of our algorithm by solving two practical instances of bilevel programs.
\end{abstract}


\section{Introduction}

Bilevel programs are optimization problems in which some variables are solutions to optimization problems themselves.  Let $x\in\mathbb{R}^n$ and $y\in\mathbb{R}^m$ be vectors, and consider the following (optimistic) bilevel programming problem:
\begin{equation*}\tag*{\hbox{$\mathsf{BLP}$}}
\begin{aligned}
\min_{x,y}\ & F(x,y)\\
\text{s.t. } & G(x) \leq 0\\
& y \in \arg\min_y \{f(x,y)\ |\ g(x,y) \leq 0\}
\end{aligned}
\end{equation*}
where $F,f$ are scalar-valued and $G,g$ are vector-valued functions.  (Equality constraints $\mathfrak{G}(x) = 0$ or $\mathfrak{g}(x,y) = 0$ are included by replacement with $\mathfrak{G}(x) \leq 0$, $\mathfrak{G}(x) \geq 0$ or $\mathfrak{g}(x,y) \leq 0$, $\mathfrak{g}(x,y) \geq 0$.) If we call $x$ the upper-level decision variables and $y$ the lower-level decision variables, then $\min \{f(x,y)\ |\ g(x,y) \leq 0\}$ is the lower level problem.  


Optimization problems with the generic form given in \textsf{BLP} are found in a variety of control applications, including control of vehicle traffic networks \cite{aswani2011,bonifaci2010,krichene2014,sharma2007,swamy2012}, inverse reinforcement learning and inverse optimization \cite{aswani2015,bertsimas2013,keshavarz2011,ng2000}, and robust control for human-automation systems \cite{vasudevan2012safe,sadigh2016planning,mintz2017behavioral}.  A solution approach for \textsf{BLP} is to replace the lower level problem by some optimality conditions and then solve the reformulated problem.  But existing algorithms \cite{anitescu2005,fukushima1999,miguel2004,outrata1990,ye1995} suffer from numerical issues \cite{kanzow2014,lin2005,scholtes2001,steffensen2010}, and so the development of new algorithms to solve \textsf{BLP} is an important area for research.

\subsection{Existing Solution Approaches}

One method \cite{anitescu2005,fukushima1999,miguel2004} for solving \textsf{BLP} replaces the lower level problem with its \textsf{KKT} conditions, giving a mathematical program with equilibrium constraints (\textsf{MPEC}). The advantage of this approach is the reformulated problem can be solved using standard nonlinear optimization software.  However, it uses complimentarity constraints, which implies a combinatorial nature to the reformulated optimization problem and leads to numerical difficulties \cite{kanzow2014,lin2005,scholtes2001,steffensen2010}.



Another method \cite{outrata1990,ye1995} for solving \textsf{BLP} replaces the lower level problem by $f(x,y) \leq \varphi(x)$ and $g(x,y) \leq 0$, where $\varphi(x) = \min_y \{f(x,y)\ |\ g(x,y) \leq 0\}$ is the value function.  This introduces a non-differentiable constraint $f(x,y) - \varphi(x) \leq 0$ (since the value function is not differentiable), and so numerical solution needs specialized algorithms that implicitly smooth the value function \cite{lin2014}.  This precludes use of standard nonlinear optimization software. 

\subsection{Duality-Based Solution Approach}

This paper develops a duality approach to solving bilevel programs with a convex lower level.  The idea is to replace the lower level problem with $f(x,y) \leq h(\lambda,x)$, $\lambda \geq 0$, and $g(x,y) \leq 0$, where $h(\lambda,x)$ is a dual function.  Under conditions with zero duality gap, these constraints force $y$ to be a minimizer of the lower level problem.  We proposed a duality approach in a paper on inverse optimization with noisy data \cite{aswani2015}, though the prior formulation is not differentiable because of the use of Lagrangian duals.  This paper constructs an alternative dual that is differentiable.  We also study constraint qualification, which was not previously considered in \cite{aswani2015}.


Our reformulation of \textsf{BLP} is such that each term is differentiable, constraint qualification holds after regularization, and the regularization is consistent in the sense as the amount of regularization is decreased than the solution of the regularized problem approaches the solution of \textsf{BLP}.  These features allow numerical solution of our reformulation (and \textsf{BLP}) using standard nonlinear optimization software.  Most of this paper focuses on technical properties of the new dual function and of the reformulation of \textsf{BLP} using this dual, and these results are used to provide theoretical justification for the algorithm that we propose for solving \textsf{BLP}.

\subsection{Outline}


Section \ref{section:prelim} provides preliminaries, including notation and our technical assumptions about \textsf{BLP}.  Section \ref{section:cldf} defines a new dual function whose maximizers are equivalent to those of the Lagrangian dual function.  Our dual is differentiable, unlike the Lagrangian dual (which is only directionally differentiable).  We use our dual to define a duality-based reformulation (\textsf{DBP}) of \textsf{BLP} in Section \ref{section:dbref}, and the equivalence of \textsf{DBP} and \textsf{BLP} is proved.  Next, we consider constraint qualification and consistency of approximation of regularized versions of \textsf{DBP}.  In Section \ref{section:numerics}, we propose an algorithm for solving \textsf{BLP} and demonstrate its effectiveness by solving two instances of practical bilevel programs.

\section{Preliminaries}
\label{section:prelim}

We define some notation and concepts from variational analysis \cite{rockafellar2009}, and then we state our assumptions about \textsf{BLP}.  

\subsection{Notation}

Let $\|\cdot\|$ be the $\ell_2$ norm.  We use $\subseteq$,$\supseteq$ for subsets and supersets, respectively.  All functions are extended real-valued, and the set $\mathcal{C}^2$ contains all twice continuously differentiable functions. Let $C$ be a set.  Then $\interior(C)$ is the interior of $C$, and the indicator function $\delta_C(x)$ is: $\delta_C(x) = 0$ if $x\in C$, and $\delta_C(x) = +\infty$ if $x\notin C$.  $\widehat{N}_C(x)$ is the regular normal cone of $C$ at $x$, and recall $v\in\widehat{N}_C(x)$ if $\langle v,x'-x\rangle \leq o(\|x'-x\|)$ for all $x'\in C$.  The normal cone of $C$ at $x$ is $N_C(x)$, and note $v\in N_C(x)$ if there are sequences $x^\nu\rightarrow x$ with $x^\nu\in C$, and $v^\nu\rightarrow v$ with $v^\nu\in\widehat{N}_C(x^\nu)$.  The non-negative orthant $\Lambda = \{\lambda : \lambda \geq 0\}$ is a closed, convex set; and its (regular) normal cone is $N_{\Lambda}(\lambda) = \{x \leq 0 : \lambda_ix_i = 0\}$.  

Now let $f$ be a function. A vector $v$ is a regular subgradient at $\overline{x}$, written $v\in\widehat{\partial}f(\overline{x})$, if $f(x) \geq f(\overline{x}) +\langle v,x-\overline{x}\rangle + o(\|x-\overline{x}\|)$.  A vector $v$ is a subgradient at $\overline{x}$, written $v\in\partial f(\overline{x})$, if there are sequences $x^\nu\rightarrow\overline{x}$ and $v^\nu\in\widehat{\partial}f(x^\nu)$ with $v^\nu\rightarrow v$.  These are used to define constraint qualification \cite{hiriart1979}.  For a constraint set $\mathcal{C} = \{g(x)\leq 0\}$, let $\overline{x}\in\mathcal{C}$ and let $I = \{i : g_i(\overline{x}) = 0\}$ be the indices of active constraints.  This $\mathcal{C}$ satisfies linear independence constraint qualification (\textsf{LICQ}) at $\overline{x}$ if all choices of $v_i\in\partial g_i(\overline{x})$, for all $i\in I$, are linearly independent.  This $\mathcal{C}$ satisfies Mangasarian-Fromovitz constraint qualification (\textsf{MFCQ}) at $\overline{x}$ if there is a $d$ such that for all choices of $v_i\in\partial g_i(\overline{x})$, for all $i\in I$, we have $v_i^\textsf{T}d< 0$.


\subsection{Technical Results}

Our first result generalizes the boundedness theorem to set-valued mappings.  Because of the technical peculiarities of continuity for set-valued mappings, we require additional assumptions beyond continuity.

\begin{lemma}
\label{lemma:btsvm}
Let $X$ be a compact set, and consider a set-valued mapping $S(x)$ that is convex-valued, continuous, and bounded for each $x\in X$.  Then $S(X)$ is bounded.
\end{lemma}

\begin{proof}
Suppose $S(X)$ is not bounded.  Then there exist sequences $x^\nu\in X$ and $s^\nu \in S(x^\nu)$ such that $\|s^\nu\|\rightarrow\infty$.  Since $X$ is compact, there is some convergent subsequence by the Bolzano-Weierstrass theorem; and so by extracting this subsequence we can assume $x^\nu\rightarrow \overline{x}$ for some $\overline{x}\in X$.  Now consider the sequence $s^\nu/\|s^\nu\|$; note the norm of each term is 1.  Hence there is some convergent subsequence, and so by extracting this subsequence we can assume $s^\nu/\|s^\nu\|\rightarrow w$ for some $w\neq 0$.  Next choose any $t\in S(\overline{x})$, and note that by continuity of $S$ there exists $t^\nu \in S(x^\nu)$ such that $t^\nu\rightarrow t$.  For any $\tau \geq 0$, there is a $\nu$ large enough such that $\tau/\|s^\nu\| < 1$.  But $S$ is convex-valued, meaning $(1-\tau/\|s^\nu\|)\cdot t^\nu + \tau/\|s^\nu\|\cdot s^\nu \in S(x^\nu)$ for $\nu$ large enough.  Taking the limit, we have $t + \tau w \in S(\overline{x})$.  This is a contradiction since: $w \neq 0$, $\tau \geq 0$ is arbitrary, and $S$ is bounded at $\overline{x}\in X$.  Thus, we have shown by contradiction that $S(X)$ is bounded.
\end{proof}

\subsection{Assumptions} 

For the lower level problem of \textsf{BLP}, we define its value function $\varphi(x) = \textstyle\min_y\{f(x,y)\ |\ g(x,y) \leq 0\}$, solution set $s(x) = \textstyle\arg\min_y\{f(x,y)\ |\ g(x,y) \leq 0\}$, and feasible set $\phi(x) = \{y : g(x,y) \leq 0\}$.  The Lagrangian dual function (\textsf{LDF}) is $\psi(\lambda,x) = \inf_y f(x,y) + \lambda^\textsf{T} g(x,y)$.  

We also make some assumptions about \textsf{BLP}.  Not all assumptions are used in every result, but we list all of them here for conciseness.  Let $X = \{x : G(x) \leq 0\}$.   Our first set of assumptions relate to the lower level problem of \textsf{BLP}.\vspace{0.5em}

\noindent\textbf{A1. } The functions $f(x, y)$, $g(x, y)$ are convex in $y$ (for fixed $x$) and satisfy $f,g\in\mathcal{C}^2$. \vspace{0.5em}

\noindent\textbf{A2. } There exists a compact, convex set $Y$ such that $\{y : \exists x \in X \text{ s.t. } g(x,y) \leq 0\}\subseteq\interior(Y)$.\vspace{0.5em}

\noindent\textbf{R1. } For each $x \in X$, there exists $y$ such that $g(x, y) < 0$.\vspace{0.5em}

\noindent The above ensure the lower level problem and its Lagrange dual problem are solvable, meaning the minimum (maximum, respectively) is attained and the set of optimal solutions is nonempty and compact.  The pointwise \textbf{R1} ensures \textsf{BLP} has a solution under the additional assumptions below.

Our next assumptions concern \textsf{BLP}, and they ensure smoothness in the objective function of \textsf{BLP} and regularity in the constraints $G(x) \leq 0$.  These conditions, when combined with the previous conditions, ensure \textsf{BLP} has a solution.\vspace{0.5em}

\noindent\textbf{A3. } The functions $F(x,y)$, $G(x)$ are twice continuously differentiable; or equivalently that $F,G \in\mathcal{C}^2$.\vspace{0.5em}

\noindent\textbf{R2. } The set $X$ is compact and nonempty, and $G(x)$ satisfies \textsf{MFCQ} for each $x\in X$.\vspace{0.5em}


\section{Constrained Lagrangian Dual Function}

\label{section:cldf}

The numerical issue with the Lagrangian dual function (\textsf{LDF}) is that it is generally nondifferentiable in $\lambda$.

\begin{example}
\label{ex:lpdirdif}
The example of linear programming is classical: Let $A \in\mathbb{R}^{p\times m}$, $b\in\mathbb{R}^p$, $c\in\mathbb{R}^m$, and define $f(x,y) = c^\textsf{T}y$ and $g(x,y) = Ax-b$.  Then, the \textsf{LDF} is
\begin{equation}
\label{eqn:linldf}
\psi(\lambda,x) = \begin{cases} -b^\textsf{T}\lambda, & \text{if } A^\textsf{T}\lambda = -c \text{ and } \lambda \geq 0\\-\infty, &\text{otherwise}\end{cases}
\end{equation}
For $\lambda_0$ such that $A^\textsf{T}\lambda_0 =- c$ and $\lambda_0\geq 0$, this \textsf{LDF} is directionally differentiable in directions $d$ such that $A^\textsf{T}d =0$ and $\lambda_0+td\geq 0$ for $t > 0$ small enough.  However, this \textsf{LDF} is not differentiable because it is discontinuous in directions $d$ such that $A^\textsf{T}d \neq 0$ or $\lambda_0 + td \ngeq 0$ for any $t > 0$.\hfill$\blacklozenge$  
\end{example}

The nondifferentiability of the \textsf{LDF} limits its utility in reformulating bilevel programs because in general closed-form expressions for the domain of the \textsf{LDF} are not available.  In this section, we construct an alternative dual function that is designed to be differentiable while retaining the saddle point and strong duality properties of the \textsf{LDF}.

\subsection{Definition and Solution Properties}

Our approach is to perform a partial dualization.  Define the Constrained Lagrangian Dual Function (\textsf{CDF}) to be
\begin{equation}
\textstyle h(\lambda,x) = \min_y \{f(x,y) + \lambda^\textsf{T}g(x,y)\ |\ y\in Y\}.
\end{equation}
The difference as compared to the (classical) \textsf{LDF} is the domain of minimization of the Lagrangian $\mathcal{L}(x,y,\lambda) = f(x,y) + \lambda^\textsf{T}g(x,y)$.  The \textsf{LDF} is the infimum of the Lagrangian over $\mathbb{R}^m$, while the \textsf{CDF} is the minimum of the Lagrangian over a compact, convex set $Y$ that contains $\{y : \exists x \in X \text{ s.t. } g(x,y) \leq 0\}$ strictly within its interior.

An important feature of the \textsf{CDF} is it maintains the strong duality of the \textsf{LDF}, and its solutions are a saddle point to the Lagrangian $\mathcal{L}(x,y,\lambda)$.  Our first result establishes an equivalence between solutions of the \textsf{CDF} and \textsf{LDF}.  

\begin{theorem}
\label{thm:eqrel}
Suppose $\mathbf{A1},\mathbf{A2}$ and $\mathbf{R1}$ hold.  Then $\arg\max_\lambda \{\psi(\lambda,x)\ |\ \lambda\geq 0\}$ is non-empty and compact, $\max_\lambda \{h(\lambda,x)\ |\ \lambda\geq 0\} = \max_\lambda \{\psi(\lambda,x)\ |\ \lambda\geq 0\}$, and $\arg\max_\lambda \{h(\lambda,x)\ |\ \lambda\geq 0\} = \arg\max_\lambda \{\psi(\lambda,x)\ |\ \lambda\geq 0\}$
\end{theorem}

\begin{proof}
We associate (see Example 11.46 of \cite{rockafellar2009}) the generalized Lagrangian $l(x,y,\lambda) = f(x,y) + \lambda^\textsf{T}g(x,y) - \delta_\Lambda(\lambda)$ for $\min_{y} \{f(x,y)\ |\ g(x,y) \leq 0\}$.  But $\psi(\lambda,x) = \inf_y l(x,y,\lambda)$ for $\lambda \geq 0$, and so $\partial_y l(x,y,\lambda) = \nabla_y f(x,y) + \nabla_y g(x,y)'\lambda$ and $\partial_\lambda [-l](x,y,\lambda) = -g(x,y) + N_\Lambda(\lambda)$.  Since $s(x)$ is compact and nonempty by Example 1.11 of \cite{rockafellar2009}, let $y^*\in s(x)$.  Theorem 11.50 and Corollary 11.51 of \cite{rockafellar2009} give: $\lambda^* \in \arg\max_\lambda \{\psi(\lambda,x)\ |\ \lambda\geq 0\}$ exists, $\psi(\lambda^*,x) = f(x,y^*)$, and $0 \in \partial_y l(x,y^*,\lambda^*)$ and $0 \in \partial_\lambda [-l](x,y^*,\lambda^*)$.  Next associate a generalized Lagrangian to the optimization problem $\min_{y\in Y} \{f(x,y)\ |\ g(x,y) \leq 0\}$.  From Example 11.46 of \cite{rockafellar2009}, its generalized Lagrangian is $\ell(x,y,\lambda) = \delta_Y(y) + f(x,y) + \lambda^\textsf{T}g(x,y) - \delta_\Lambda(\lambda)$.  Note $h(\lambda,x) = \min_y \{\ell(x,y,\lambda)\ |\ y\in Y\}$ for $\lambda \geq 0$, $\partial_y \ell(x,y,\lambda) = N_Y(y) + \nabla_y f(x,y) + \nabla_y g(x,y)'\lambda$, and $\partial_\lambda [-\ell](x,y,\lambda) = -g(x,y) + N_\Lambda(\lambda)$.  Since $0 \in N_Y(y)$ and $0 \in \partial_y l(x,y^*,\lambda^*)$, we have $0 \in \partial_y \ell(x,y^*,\lambda^*)$.  Similarly, $0 \in \partial_\lambda [-l](x,y^*,\lambda^*)$ yields $0 \in \partial_\lambda [-\ell](x,y^*,\lambda^*)$.  Thus, we can apply Theorem 11.50 and Corollary 11.51 of \cite{rockafellar2009}, which gives: $\lambda^* \in \arg\max_\lambda \{h(\lambda,x)\ |\ \lambda\geq 0\}$, and $h(\lambda^*,x) = f(x,y^*)$.  So $h(\lambda^*,x) = \psi(\lambda^*,x)$, proving the first part of the result.

But recall that $\lambda^* \in \arg\max_\lambda \{h(\lambda,x)\ |\ \lambda\geq 0\}$.  This means $\arg\max_\lambda \{h(\lambda,x)\ |\ \lambda\geq 0\} \supseteq \arg\max_\lambda \{\psi(\lambda,x)\ |\ \lambda\geq 0\}$.  Theorem 11.50 and Corollary 11.51 of \cite{rockafellar2009} give: $\mu^* \in \arg\max_\lambda \{h(\lambda,x)\ |\ \lambda\geq 0\}$ exists, $h(\mu^*,x) = f(x,y^*)$, $0 \in \partial_y \ell(x,y^*,\mu^*)$, and $0 \in \partial_\lambda [-\ell](x,y^*,\mu^*)$.  The condition $0 \in \partial_\lambda [-\ell](x,y^*,\mu^*)$ implies $0 \in \partial_\lambda [-l](x,y^*,\mu^*)$ and $y \in \phi(x)$.  This second consequence implies $y \in \text{int}(Y)$ by \textbf{A2}, and hence $N_Y(y) = \{0\}$.  So $0 \in \partial_y l(x,y^*,\mu^*)$ because $0 \in \partial_y \ell(x,y^*,\mu^*)$.  Applying Theorem 11.50 and Corollary 11.51 of \cite{rockafellar2009} implies that $\mu^* \in \arg\max_\lambda \{\psi(\lambda,x)\ |\ \lambda\geq 0\}$.  Thus, $\arg\max_\lambda \{h(\lambda,x)\ |\ \lambda\geq 0\} \subseteq \arg\max_\lambda \{\psi(\lambda,x)\ |\ \lambda\geq 0\}$.  Since we have shown both set inclusions, this implies equality and hence the second result.
\end{proof}

This result is nontrivial because a slight (and subtle) relaxation of the hypothesis causes the result to become untrue.  Suppose we replace \textbf{A2} with an assumption on the existence of a compact, convex set $Z$ with $Z \supseteq \{y : \exists x \in X \text{ s.t. } g(x,y) \leq 0\}$.  (The difference from \textbf{A2} is $\{y : \exists x \in X \text{ s.t. } g(x,y) \leq 0\}$ is in the interior of $Y$, while it is only a subset of $Z$.)  The above result fails because in general we have $\arg\max_\lambda \{\eta(\lambda,x)\ |\ \lambda\geq 0\} \supseteq \arg\max_\lambda \{\psi(\lambda,x)\ |\ \lambda\geq 0\}$ for $\eta(\lambda,x) = \min_y\{f(x,y)+\lambda^\textsf{T}g(x,y)\ |\ y\in Z\}$.  The following example provides one situation where this superset is proper, and this emphasizes the importance of \textbf{A2}.

\begin{example}
\label{ex:simplp}
Consider: $f(x,y) = y$, $g_1(x,y) = -y-1$, and $g_2(x,y) = y-1$.  If $Z = \phi(x) = \{y : y\in[-1,1]\}$, then $\eta(\lambda,x) = -|1-\lambda_1+\lambda_2|-\lambda_1-\lambda_2$ and 
\begin{equation}
\psi(\lambda,x) = \begin{cases} -\lambda_1-\lambda_2, &\text{if } -\lambda_1 + \lambda_2 = -1,\ \lambda \geq 0\\-\infty, &\text{otherwise}\end{cases}
\end{equation}
Thus $\arg\max_\lambda\{\psi(\lambda,x)\, |\, \lambda\geq 0\} = \{\lambda : \lambda_1 = 1, \lambda_2 = 0\}$, $\arg\max_\lambda\{\eta(\lambda,x)\, |\, \lambda\geq 0\} = \{\lambda : \lambda_1 \in [0,1], \lambda_2 = 0\}$, and $\arg\max_\lambda \{\eta(\lambda,x)\, |\, \lambda\geq 0\} \supset\arg\max_\lambda \{\psi(\lambda,x)\, |\, \lambda\geq 0\}$.\hfill$\blacklozenge$
\end{example}

Because the \textsf{CDF} is constructed to have the same solutions as the \textsf{LDF}, the \textsf{CDF} enjoys the same strong duality and saddle point properties of the \textsf{LDF}.  

\begin{corollary}
\label{cor:sadpoint}
Suppose $\mathbf{A1},\mathbf{A2}$ and $\mathbf{R1}$ hold.  If we have that $\lambda^*\in\arg\max_\lambda\{h(\lambda,x)\ |\ \lambda \geq 0\}$ and $y^*\in s(x)$, then $\min_y \{f(x,y)\ |\ g(x,y) \leq 0\} = \max_\lambda \{h(\lambda,x)\ |\ \lambda\geq 0\} = \mathcal{L}(x,y^*,\lambda^*)$ and $\mathcal{L}(x,y^*,\lambda) \leq \mathcal{L}(x,y^*,\lambda^*)\leq\mathcal{L}(x,y,\lambda^*)$ for all $y\in\mathbb{R}^m$ and $\lambda \geq 0$.
\end{corollary}

\begin{proof}
Theorem \ref{thm:eqrel} implies $\arg\max_\lambda \{h(\lambda,x)\ |\ \lambda\geq 0\} = \arg\max_\lambda \{\psi(\lambda,x)\ |\ \lambda\geq 0\}$, so $\lambda^*\in\arg\max_\lambda\{\psi(\lambda,x)\ |\ \lambda \geq 0\}$. Theorem 11.50 and Corollary 11.51 of \cite{rockafellar2009} give $\min_y \{f(x,y)\ |\ g(x,y) \leq 0\} = \max_\lambda \{h(\lambda,x)\ |\ \lambda\geq 0\} = l(x,y^*,\lambda^*)$ and $l(x,y^*,\lambda) \leq l(x,y^*,\lambda^*)\leq l(x,y,\lambda^*)$ for all $y\in\mathbb{R}^m$ and $\lambda \geq 0$, where $l(x,y,\lambda)$ is the generalized Lagrangian in the proof of Theorem \ref{thm:eqrel}.  But $\mathcal{L}(x,y,\lambda) = l(x,y,\lambda)$ when $\lambda \geq 0$.
\end{proof}

\textbf{A2} is again crucial, and the result does not hold if it is relaxed using the set $Z$ defined above. The saddle point result (i.e., the second part of the corollary) fails for $\mathcal{L}$.  (However, a saddle point result holds for the generalized Lagrangian $\ell(x,y,\lambda)$ defined in the proof of Theorem \ref{thm:eqrel}.)  The following continuation of the previous example shows this.

\addtocounter{example}{-1}
\begin{example}[continued]
If $Y = \{y : y\in[-2,2]\}$, then $\eta(\lambda,x) = -|1-\lambda_1+\lambda_2|-\lambda_1-\lambda_2$ and $\arg\max_\lambda\{\eta(\lambda,x)\ |\ \lambda\geq 0\} = \{\lambda : \lambda_1 \in [0,1] \text{ and } \lambda_2 = 0\}$.  Choosing $y = -2$ and $\lambda^*=0$ gives $\mathcal{L}(x,y,\lambda^*) = -2 < \mathcal{L}(x,y^*,\lambda^*) = -1$ because $y^*=-1$.  So maximizers of $\eta(\lambda,x)$ (which uses $Z$) do not satisfy the saddle point property for $\mathcal{L}$.  In contrast, note $h(\lambda,x) = -2\cdot|1-\lambda_1+\lambda_2|-\lambda_1-\lambda_2$.  A simple calculation gives $\arg\max_\lambda\{h(\lambda,x)\ |\ \lambda\geq 0\} = \{\lambda : \lambda_1 = 1 \text{ and } \lambda_2 = 0\}$ (matching Theorem \ref{thm:eqrel} since $\arg\max_\lambda\{\psi(\lambda,x)\ |\ \lambda\geq 0\} = \{\lambda : \lambda_1 = 1 \text{ and } \lambda_2 = 0\}$).  Thus, the solution provided by $h(\lambda,x)$ gives $\mathcal{L}(x,y,\lambda^*) = -1 \geq \mathcal{L}(x,y^*,\lambda^*) = -1 \geq \mathcal{L}(x,y^*,\lambda) = -1-2\lambda_2$ for all $y\in\mathbb{R}^m$ and $\lambda\geq 0$, which matches Corollary \ref{cor:sadpoint}.\hfill$\blacklozenge$
\end{example}

\addtocounter{example}{1}

\subsection{Differentiability}

The distinguishing property of the \textsf{CDF} is that it is differentiable, while the \textsf{LDF} is only directionally differentiable (see Example \ref{ex:lpdirdif}).  The differentiability occurs because the \textsf{CDF} is defined as a minimization over a compact set that is independent of $\lambda,x$.  In particular, if we define $\textstyle\sigma(\lambda,x) = \arg\min_y \{f(x,y)+\lambda^\textsf{T}g(x,y)\ |\ y\in Y\}$, then we can state the differentiability of the \textsf{CDF}.



\begin{theorem}
\label{thm:diffsing}
Suppose $\mathbf{A1},\mathbf{A2}$ and $\mathbf{R1}$ hold.  If $(\lambda,x)$ is such that $\sigma(\lambda,x)$ is singleton; then the $\mathsf{CDF}$ is differentiable at $(\lambda,x)$, and its gradient is given by
\begin{equation}
\label{eqn:gradone}
\begin{aligned}
&\nabla_x h(\lambda,x) = \nabla_x f(x,\overline{y}) + \lambda^\textsf{T}\nabla_x g(x, \overline{y})\\
&\nabla_\lambda h(\lambda,x)= g(x,\overline{y})
\end{aligned}
\end{equation}
where we have that $\{\overline{y}\} = \sigma(\lambda,x)$.
\end{theorem}

\begin{proof}
This follows from Theorem 4.13 and Remark 4.14 of \cite{bonnans2000}. 
\end{proof}

Though determining if $\sigma(\lambda,x)$ is singleton can be difficult, a simple-to-check condition ensures this is always the case:

\begin{corollary}
\label{cor:strtconv}
Suppose $\mathbf{A1},\mathbf{A2}$ and $\mathbf{R1}$ hold.  If $\lambda \geq 0$ and $f(x,y)$ is strictly convex in $y$ for every $x\in X$; then the $\mathsf{CDF}$ is differentiable at $(\lambda,x)$, and its gradient is given in (\ref{eqn:gradone}), where we have that $\{\overline{y}\} = \sigma(\lambda,x)$.
\end{corollary}

\begin{proof}
Since $\lambda \geq 0$, $f(x,y)+\lambda^\textsf{T}g(x,y)$ is strictly convex in $y$ for every $x \in X$ (see for instance Exercise 2.18 in \cite{rockafellar2009}).  Example 1.11 and Theorem 2.6 of \cite{rockafellar2009} imply $\sigma(\lambda,x)$ is singleton.  We can then apply Theorem \ref{thm:diffsing}.
\end{proof}

For the case where $f(x,y)$ is not strictly convex, we can define a regularized \text{CDF} that is guaranteed to be differentiable.  In particular, we define the regularized constrained Lagrangian dual function (\textsf{RDF}) to be 
\begin{multline}
h_\mu(\lambda,x) = \\\textstyle\min_y \{\mu\|y\|^2 + f(x,y) + \lambda^\textsf{T}g(x,y)\ |\ y\in Y\},
\end{multline}
where $\mu \geq 0$.  We can interpret this as the \textsf{CDF} for an optimization problem where the objective has been changed to $\mu\|y\|^2 + f(x,y)$.  The benefit of adding the $\mu\|y\|^2$ term is it makes the objective of the optimization problem defining $h_\mu(\lambda,x)$ strictly convex, and therefore ensures the \textsf{RDF} is differentiable as long as $\mu > 0$.  More formally, if $\textstyle\sigma_\mu(\lambda,x) = \arg\min_y \{\mu\|y\|^2+f(x,y)+\lambda^\textsf{T}g(x,y)\ |\ y\in Y\}$, then:

\begin{corollary}
\label{cor:difmu}
Suppose $\mathbf{A1},\mathbf{A2}$ and $\mathbf{R1}$ hold.  If $\lambda \geq 0$ and $\mu > 0$; then the $\mathsf{RDF}$ is differentiable at $(\lambda,x)$, and its gradient is given by
\begin{equation}
\label{eqn:defnab}
\begin{aligned}
&\nabla_x h_\mu(\lambda,x) = \nabla_x f(x,\overline{y}) + \lambda^\textsf{T}\nabla_x g(x, \overline{y})\\
&\nabla_\lambda h_\mu(\lambda,x)= g(x,\overline{y})
\end{aligned}
\end{equation}
where we have that $\{\overline{y}\} = \sigma_\mu(\lambda,x)$.
\end{corollary}

\begin{proof}
Since $\|y\|^2$ is strictly convex and $f(x,y)$ is convex, $\mu\|y\|^2+f(x,y)$ is strictly convex in $y$ for every $x \in X$ (Exercise 2.18 in \cite{rockafellar2009}).  So Corollary \ref{cor:strtconv} applies.
\end{proof}

More generally, both the \textsf{CDF} and \textsf{RDF} have a strong type of regularity because of their construction.  This regularity will be useful for proving subsequent results.  

\begin{proposition}
\label{prop:hmulc2}
Suppose $\mathbf{A1},\mathbf{A2}$ and $\mathbf{R1}$ hold.  Then for $\mu \geq 0$, we have $[-h]_\mu(\lambda,x)$ is locally Lipschitz continuous; and its subgradient is nonempty, compact, and given by
\begin{equation}
\label{eqn:subgradhmu}
\begin{aligned}
&\partial_x [-h]_\mu(\lambda,x) = -\mathrm{co}\{\nabla_x f(x,\overline{y}) + \\
&\hspace{3.3cm}\lambda^\textsf{T}\nabla_x g(x, \overline{y})\rangle\ |\ \overline{y} \in \sigma_\mu(\lambda,x)\}\\
&\partial_\lambda [-h]_\mu(\lambda,x) = -\mathrm{co}\{g(x,\overline{y})\ |\ \overline{y} \in \sigma_\mu(\lambda,x)\}
\end{aligned}
\end{equation}
where we have that $\sigma_\mu(\lambda,x) = \arg\min_y \{\mu\|y\|^2+f(x,y)+\lambda^\textsf{T}g(x,y)\ |\ y\in Y\}$.
\end{proposition}

\begin{proof}
Note $[-h]_\mu(\lambda,x) = \max_y \{-\mu\|y\|^2 - f(x,y) - \lambda^\textsf{T}g(x,y)\rangle\ |\ y\in Y\}$, by rewriting the definition of $h_\mu(\lambda,x)$.  So $[-h]_\mu$ is lower-$\mathcal{C}^2$ by definition (see \cite{rockafellar2009,rockafellar1982}).  This implies local Lipschitz continuity \cite{rockafellar2009,rockafellar1982}.  Theorem 9.13 of \cite{rockafellar2009} gives nonemptiness and compactness of the subgradient, and the formula (\ref{eqn:subgradhmu}) is due to Theorem 2.1 of \cite{clarke1975}.
\end{proof}

\subsection{Convergence Properties}

An important aspect of the \textsf{RDF} is it epi-converges to the \textsf{CDF} as $\mu \rightarrow 0$.  Note this convergence does not require $\sigma(\lambda,x)$ to be singleton, and hence applies even when $f(x,y)$ is not strictly convex in $y$ for every $x \in X$.  Also, note the epi-convergence result applies to $-h(\lambda,\mu)$ and $-h_\mu(\lambda,\mu)$ since we are typically concerned with maximizing the dual.


\begin{proposition}
\label{prop:epicon}
Suppose $\mathbf{A1},\mathbf{A2}$ and $\mathbf{R1}$ hold.  Then the function $[-h]_\mu(\lambda, x)$ is pointwise decreasing in $\mu$, and we have that $\displaystyle\elim_{\mu\rightarrow 0} [-h]_\mu(\lambda,x) = [-h](\lambda,x)$.
\end{proposition}

\begin{proof}
The Berge maximum theorem \cite{berge1963} implies $h(\lambda,x)$ and $h_{\mu}(\lambda,x)$ are continuous (for each fixed $\mu > 0$). Second, note Proposition 7.4.c of \cite{rockafellar2009} gives that for fixed $\lambda,x$ we have $\elim_{\mu\rightarrow 0} \mu\|y\|^2 + f(x,y) + \lambda^\textsf{T}g(x,y) = f(x,y) + \lambda^\textsf{T}g(x,y)$.  And so by Theorem 7.33 of \cite{rockafellar2009}, we have for fixed $\lambda,x$ that $\lim_{\mu\rightarrow 0} h_\mu(\lambda,x) = h(\lambda,x)$.  Now let $\overline{y}\in\sigma_\mu(\lambda,x)$, and observe that for $0\leq \mu_1 < \mu_2$ we have $h_{\mu_1}(\lambda, x) \leq \mu_1\|\overline{y}\|^2 + f(x,\overline{y}) + \lambda^\textsf{T}g(x,\overline{y})\leq \mu_2\|\overline{y}\|^2 + f(x,\overline{y}) + \lambda^\textsf{T}g(x,\overline{y})= h_{\mu_2}(\lambda, x)$.  Thus, $[-h]_\mu(\lambda, x)$ is decreasing in $\mu$.  This implies $\sup_{\mu > 0} [-h]_\mu(\lambda, x) = [-h](\lambda, x)$ since from above we have $\lim_{\mu\rightarrow 0} [-h]_\mu(\lambda,x) = [-h](\lambda,x)$.  Finally, using Proposition 7.4.d of \cite{rockafellar2009} gives the desired result: $\elim_{\mu\rightarrow 0} [-h]_\mu(\lambda,x) = [-h](\lambda,x)$.
\end{proof}

\section{Duality-Based Reformulation}
\label{section:dbref}

It will be more convenient to work with the approximate bilevel programming problem, which is defined as
\begin{equation*}\tag*{\hbox{$\mathsf{BLP}(\epsilon)$}}
\begin{aligned}
\min_{x,y}\ & F(x,y)\\
\text{s.t. } & G(x) \leq 0\\
& y \in \epsilon\text{-}\arg\textstyle\min_y \{f(x,y)\ |\ g(x,y) \leq 0\}
\end{aligned}
\end{equation*}
where $y \in \epsilon\text{-}\arg\textstyle\min_y \{f(x,y)\ |\ g(x,y) \leq 0\}$ means $f(x,y) \leq \min_y \{f(x,y)\ |\ g(x,y) \leq 0\}+\epsilon$ and $g(x,y) \leq \epsilon$.  (Equivalently, we have that $y$ is an $\epsilon$-solution in the sense of \cite{nemirovski2004,nesterov1994}.) This problem is equivalent to \textsf{BLP} when $\epsilon = 0$.

We first define our duality-based reformulation of \textsf{BLP}$(\epsilon)$, and then show its equivalence to the approximate bilevel program.  Next we study constraint qualification of our reformulation and provide conditions that ensure \textsf{MFCQ} holds.  Since the duality-based reformulation has regularization, we conclude by providing sufficient conditions that ensure convergence of solutions to the regularized duality-based reformulation to solutions of the limiting problem.

\subsection{Definition}

Our duality-based reformulation of \textsf{BLP}$(\epsilon)$ using \textsf{RDF} is
\begin{equation*}\tag*{\hbox{$\mathsf{DBP}(\epsilon,\mu)$}}
\begin{aligned}
\min_{x,y,\lambda}\ & F(x,y)\\
\text{s.t. } & (x,y,\lambda)\in\mathcal{C}(\epsilon,\mu)
\end{aligned}
\end{equation*}
where the feasible set of $\mathsf{DBP}(\epsilon,\mu)$ is given by
\begin{equation}
\mathcal{C}(\epsilon,\mu) = \left\{(x,y,\lambda) : \begin{aligned}& G(x) \leq 0,\, g(x,y) \leq \epsilon,\, \lambda \geq 0\\
& f(x,y) - h_{\mu}(\lambda,x) \leq \epsilon
\end{aligned}\right\}
\end{equation}
One useful property of the reformulation \textsf{DBP}$(\epsilon,\mu)$ is that it is convex when $x$ is fixed, and a proof of a less general version of this result is found in Proposition 6 of \cite{aswani2015}.  

The next result shows that upper-bounding the objective by the \textsf{RDF}, which is done in the feasible set of \textsf{DBP}$(\epsilon,\mu)$, is an optimality condition for the lower level problem. 
\begin{proposition}
\label{prop:epssol}
Suppose $\mathbf{A1},\mathbf{A2}$ and $\mathbf{R1}$ hold.  Then a point $y$ is an $\epsilon$-solution to the lower level problem if and only if there exists $\lambda$ such that $(x,y,\lambda) \in \mathcal{C}(\epsilon,0)$.  If $\mu \geq 0$ and a point $y$ is an $\epsilon$-solution to the lower level problem, then there exists $\lambda$ such that $(x,y,\lambda)\in\mathcal{C}(\epsilon,\mu)$.
\end{proposition}

\begin{proof} 
By Proposition 5 of \cite{aswani2015} a point $y$ is an $\epsilon$-solution to the lower level problem if and only if there exists $\lambda$ such that the following inequalities are satisfied: $f(x,y) - \psi(\lambda,x) \leq \epsilon$, $g(x,y) \leq \epsilon$, $\lambda \geq 0$.  The result holds if we can show there exists $\lambda' \geq 0$ such that $f(x,y) - \psi(\lambda',x) \leq \epsilon$ if and only if there exists $\lambda'' \geq 0$ such that $f(x,y) - h_0(\lambda'',x) \leq \epsilon$.  Let $\lambda''\in\arg\max_\lambda\{h_0(\lambda,x)\ |\ \lambda \geq 0\}$, and note $f(x,y) - h_0(\lambda'',x) \leq f(x,y) - \psi(\lambda',x)$ by Theorem \ref{thm:eqrel}.  Similarly, let $\lambda' \in\arg\max_\lambda\{\psi(\lambda,x)\ |\ \lambda\geq 0\}$, and note $f(x,y) - \psi(\lambda',x) \leq f(x,y) - h_0(\lambda'',x)$ by Theorem \ref{thm:eqrel}.  Next recall there exists $\lambda$ such that $(x,y,\lambda) \in \mathcal{C}(0,0)$, and so $f(x,y) - h_\mu(\lambda,x) \leq f(x,y) - h_0(\lambda,x)\leq\epsilon$ since Proposition \ref{prop:epicon} shows $[-h]_\mu(\lambda, x)$ is decreasing. 
\end{proof}

%

Out next result is on the equivalence of solutions to \textsf{BLP}$(\epsilon)$ and \textsf{DBP}$(\epsilon,0)$.  A similar result was shown in \cite{dempe2012} for the \textsf{KKT} reformulation, but we cannot apply their results to our setting because feasible $\lambda$ for \textsf{DBP}$(\epsilon)$ are not necessarily Lagrange mutlipliers when $\epsilon > 0$.

\begin{proposition}
\label{prop:globmineq}
Suppose $\mathbf{A1},\mathbf{A2}$ and $\mathbf{R1}$ hold.  A point $(\overline{x},\overline{y})$ is a minimizer of $\mathsf{BLP}(\epsilon)$ if and only if for some feasible $\lambda \geq 0$ the point $(\overline{x},\overline{y}, \lambda)$ is a minimizer of $\mathsf{DBP}(\epsilon,0)$.
\end{proposition}

\begin{proof} 
We prove this by showing $(x',y')$ is not a global minimum of \textsf{BLP}$(\epsilon)$ if and only if $(x',y',\lambda')$ is not a global minimum of \textsf{DBP}$(\epsilon,0)$ for some feasible $\lambda'\geq0$.  Suppose $(x',y')$ is not a global minimum of \textsf{BLP}$(\epsilon)$.  Then there exists $(x,y)$ feasible for \textsf{BLP}$(\epsilon)$, and with $F(x,y) < F(x',y')$.  By Proposition \ref{prop:epssol}, there exists $\lambda \geq 0$ such that $(x,y,\lambda)$ is feasible for \textsf{DBP}$(\epsilon,0)$, which implies $(x',y',\lambda')$ is not a global minimum of \textsf{DBP}$(\epsilon,0)$.  Similarly, suppose $(x',y',\lambda')$ is not a global minimum of \textsf{DBP}$(\epsilon,0)$.  Then there exists $(x,y,\lambda)$ feasible for \textsf{DBP}$(\epsilon,0)$, and such that $F(x,y) < F(x',y')$.  However, this $(x,y)$ is feasible for \textsf{BLP}$(\epsilon)$ by Proposition \ref{prop:epssol}.  Thus $(x',y')$ is not a global minimum of \textsf{BLP}$(\epsilon)$.
\end{proof}


The issue of equivalence between local minimizers of \textsf{BLP}$(\epsilon)$ and \textsf{DBP}$(\epsilon,0)$ is more complex.  The \textsf{KKT} reformulation generally lacks such an equivalence \cite{dempe2012}, and \cite{dempe2012} argues that assuming \textsf{LICQ} for the lower level problem provides equivalence of local minimizers since this ensures uniqueness (and hence continuity) of the Largrange multipliers \cite{wachsmuth2013}.  However, results for the \textsf{KKT} reformulation \cite{dempe2012} cannot be applied to our setting because feasible $\lambda$ for \textsf{DBP}$(\epsilon,0)$ are not necessarily Lagrange mutlipliers.  

\begin{proposition}
\label{prop:lseq}
Suppose $\mathbf{A1},\mathbf{A2}$ and $\mathbf{R1}$ hold.  If $\epsilon > 0$, or $g(x,y)$ satisfies $\mathsf{LICQ}$ for each $x \in X$; then a point $(\overline{x},\overline{y})$ is a local minimum of $\mathsf{BLP}(\epsilon)$ if and only if for some feasible $\lambda \geq 0$ the point $(\overline{x},\overline{y}, \lambda)$ is a local minimum of $\mathsf{DBP}(\epsilon,0)$.
\end{proposition}

\begin{proof} 
We show $(x',y')$ is not a local minimum of \textsf{BLP}$(\epsilon)$ if and only if $(x',y',\lambda')$ is not a local minimum of \textsf{DBP}$(\epsilon,0)$ for some feasible $\lambda'\geq0$.  First suppose $(x',y',\lambda')$ is not a local minimum of \textsf{DBP}$(\epsilon,0)$.  Then there exists a feasible sequence $(x^\nu,y^\nu,\lambda^\nu)\rightarrow(x',y',\lambda')$ with $F(x^\nu,y^\nu) < F(x',y')$, where $(x^\nu,y^\nu)$ is feasible for \textsf{BLP}$(\epsilon)$ by Proposition \ref{prop:epssol}.  This shows $(x',y')$ is not a local minimum of \textsf{BLP}$(\epsilon)$.  To prove the other direction, suppose $(x',y')$ is not a local minimum of \textsf{BLP}$(\epsilon)$.  Then there exists a sequence of feasible $(x^\nu,y^\nu)\rightarrow(x,y)$ with $F(x^\nu,y^\nu) < F(x',y')$.  We must now consider two cases.

The first case is when $\epsilon > 0$.  Define $\Phi_{\epsilon,\mu}(x) = \{(y,\lambda) : (x,y,\lambda)\in\mathcal{C}(\epsilon,\mu)\}$.  For each $x \in X$, choosing $\overline{y}$ to be a solution to the lower level problem (which exists by Example 1.11 of \cite{rockafellar2009} and Theorem 2.165 of \cite{bonnans2000}) gives a corresponding $\underline{\lambda}$ (by Proposition \ref{prop:epssol}) that satisfies $f(x,\overline{y}) - h_0(x,\underline{\lambda}) \leq 0 < \epsilon$.  But $h_\mu$ is decreasing in $\mu$ (Proposition \ref{prop:epicon}), and so we have $f(x,\overline{y}) - h_\mu(x,\underline{\lambda}) < \epsilon$.  Since $h_\mu$ is continuous by Proposition \ref{prop:hmulc2}, this means we can choose $\underline{\lambda}' > 0$ such that $f(x,\overline{y}) - h_\mu(x,\underline{\lambda}') < \epsilon$.  Combining this with \textbf{A1}, Proposition \ref{prop:hmulc2}, the convexity of \textsf{DBP}$(\epsilon,\mu)$ when $x$ is fixed, and Example 5.10 of \cite{rockafellar2009} shows $\Phi_{\epsilon,\mu}$ is continuous when $\epsilon > 0$.  Hence there exists a sequence $\lambda^\nu\rightarrow\lambda'$ with $(x^\nu,y^\nu,\lambda^\nu)$ feasible for \textsf{DBP}$(\epsilon,0)$.  This implies that the point $(x',y',\lambda')$ is not a local minimum of \textsf{DBP}$(\epsilon,0)$.  

The second case is when $\epsilon = 0$ and \textsf{LICQ} holds.  Theorem \ref{thm:eqrel} and Corollary \ref{cor:sadpoint} imply $\Phi_{0,0}(x)$ consists of saddle points to the Lagrangian $\mathcal{L}$, and hence satisfy the \textsf{KKT} conditions (see Corollary 11.51 of \cite{rockafellar2009}) because of the constraint qualification in \textbf{R1}.  So there is a unique $\lambda'(x)$ that makes $(x',y',\lambda'(x))$ feasible for \textsf{DBP}$(\epsilon,0)$ \cite{wachsmuth2013}. By Corollary \ref{cor:sadpoint} we have $\lambda'(x) \in \arg\max_\lambda h_0(\lambda,x)$, and so $\lambda'(x)$ is a continuous function since it is single-valued \cite{wachsmuth2013} and osc by the Berge maximum theorem \cite{berge1963}.  Hence there exists $\lambda^\nu\rightarrow\lambda'(x)$ with $(x^\nu,y^\nu,\lambda^\nu)$ feasible for \textsf{DBP}$(\epsilon,0)$.  This implies $(x',y',\lambda')$ is not a local minimum of \textsf{DBP}$(\epsilon,0)$.
\end{proof}


\subsection{Constraint Qualification}

One difficulty with solving bilevel programs is reformulations do not satisfy constraint qualification \cite{ye1995,scholtes2001,scheel2000}.  The issue is not that the feasible region of a bilevel program usually has no interior, but rather that an inequality representing optimality must fundamentally violate constraint qualification since we can interpret constraint qualification as stating the constraints have no local optima \cite{polak1997}.  However, one benefit of our regularization is it leads to constraint qualification of the regularized problem \textsf{DBP}$(\epsilon,\mu)$.  


\begin{theorem}
\label{thm:mfcq}
Suppose $\mathbf{A1}$--$\mathbf{A3}$ and $\mathbf{R1},\mathbf{R2}$ hold.  If $\epsilon > 0$, then $\mathsf{MFCQ}$ holds for $\mathsf{DBP}(\epsilon,\mu)$.
\end{theorem}

\begin{proof}
Consider any $(x,y,\lambda)$ feasible for \textsf{DBP}$(\epsilon,\mu)$.  Note some subset of the constraints $g(x,y) \leq \epsilon$, $G(x) \leq 0$, and $\lambda\geq0$ may be active, and label the indices of the active constraints by $I, J, K$.  Slater's condition holds for $g(x,y)\leq\epsilon$ by \textbf{R1}, \textsf{MFCQ} holds for $G(x)\leq0$ by \textbf{R2}, and Slater's condition holds for $-\lambda\leq 0$ since it clearly has an interior.  Since Slater's condition is equivalent to \textsf{MFCQ} for convex sets \cite{rockafellar2009}, there exists $d_x,d_y,d_\lambda$ such that $\nabla_xG_i(x)^\textsf{T}d_x < 0$, $\nabla_y g_j(x,y)^\textsf{T} d_y< 0$, and $\nabla_\lambda [-\lambda]_k^\textsf{T}d_\lambda< 0$ for active constraints $i\in I$, $j\in J$, and $k\in K$.

Next, we consider two sub-cases. The first sub-case has $f(x,y) - h_\mu(\lambda,x) < \epsilon$, which means this constraint cannot be active.  Note we can choose $\gamma > 0$ small enough to ensure $\gamma \nabla_x G_i(x)^{\textsf{T}}d_x < 0$ and $\gamma \nabla_x g_j(x,y)^\textsf{T}d_x + \nabla_y g_j(x,y)^\textsf{T}d_y < 0$ for $i\in I$ and $j\in J$, since by the Cauchy-Schwartz inequality we have $\gamma \nabla_x g_j(x,y)^\textsf{T}d_x \leq \gamma\cdot\|\nabla_x g_j(x,y)\|\cdot\|d_x\|$.  Thus, \textsf{MFCQ} holds in this sub-case.  In the second sub-case, $f(x,y) - h_\mu(\lambda,x) = \epsilon$.  Let $y^* \in \arg\min\{f(x,y)\ |\ g(x,y)\leq 0\}$ and $\lambda^* \in \arg\max\{h(\lambda,x)\ |\ \lambda \geq 0\}$, and note Corollary \ref{cor:sadpoint} gives $f(x,y^*) - h_0(\lambda^*,x) \leq 0$.  This implies $f(x,y^*) - h_0(\lambda^*,x) < \epsilon$.  But recall Proposition \ref{prop:epicon} gives that $[-h]_\mu(\lambda, x)$ is decreasing in $\mu$, and so we have $f(x,y^*) - h_\mu(\lambda^*,x) \leq f(x,y^*) - h_0(\lambda^*,x) < \epsilon = f(x,y) - h_\mu(\lambda,x).$ Consider any $v_x \in \partial_x[-h]_\mu(\lambda,x)$ and $v_\lambda \in \partial_\lambda[-h]_\mu(\lambda,x)$, where existence and boundedness of the subgradient comes from Proposition \ref{prop:hmulc2}.  Observe that $f(x,y) - h_\mu(\lambda,x)$ is convex in $y,\lambda$, and by its convexity we have $\nabla_y f(x,y)^\textsf{T}(y^* - y) + v_\lambda^\textsf{T}(\lambda^*-\lambda) \leq f(x,y^*) - h_\mu(\lambda^*,x) - f(x,y) + h_\mu(\lambda,x) < 0$.  Since the subgradient of $[-h]_\mu$ is bounded, by the Cauchy-Schwarz inequality we can choose $\gamma > 0$ small enough to ensure $\gamma \nabla_x G_i(x)^\textsf{T}d_x < 0$, $\gamma \nabla_x g_j(x,y)^\textsf{T}d_x + \nabla_y g_j(x,y)^\textsf{T}d_y < 0$, and $\gamma (\nabla_x f(x,y) + v_x)^\textsf{T}d_x + \nabla_y f(x,y)^\textsf{T}(y^* - y) + v_\lambda^\textsf{T}(\lambda^*-\lambda) < 0$, for $i\in I$ and $j\in J$.  We next compute $\nabla_\lambda [-\lambda]_k^\textsf{T}(\lambda^*-\lambda)$ for $k\in K$.  If $-\lambda_k \leq 0$ is active,  then $\lambda_k = 0$ and $\lambda^*_k - \lambda_k \geq 0$ because $\lambda^*_k \geq 0$.  Thus, $\nabla_\lambda [-\lambda]_k^\textsf{T}(\lambda^*-\lambda) \leq 0$ for $k\in K$.  Next, note we can choose $\gamma'$ such that $\nabla_\lambda [-\lambda]_k^\textsf{T}(\lambda^*-\lambda) + \gamma'\nabla_\lambda [-\lambda]_k^\textsf{T}d_\lambda < 0$ for $k\in K$ and  $\gamma(\nabla_x f(x,y) + v_x)^\textsf{T}d_x + \nabla_y f(x,y)^\textsf{T}(y^* - y) + v_\lambda^\textsf{T}(\lambda^*-\lambda) + \gamma'v_\lambda^\textsf{T}d_y < 0$.  So \textsf{MFCQ} holds in this sub-case.
\end{proof}

\subsection{Consistency of Approximation}


We show the regularized problems \textsf{DBP}$(\epsilon,\mu)$ are consistent approximations \cite{polak1997,royset2016} of the limiting problem \textsf{DBP}$(\overline{\epsilon},0)$ under appropriate conditions.  Our first result concerns convergence of the constraint sets $\mathcal{C}(\epsilon,\mu)$, which leads as a corollary to convergence of optimizers of the regularized problems to optimizers of the limiting problem.

\begin{proposition}
\label{prop:convcons}
Suppose $\mathbf{A1}$--$\mathbf{A3}$ and $\mathbf{R1}$ hold.  Then for any $\overline{\epsilon} \geq 0$ we have that $\lim_{\epsilon\downarrow\overline{\epsilon},\mu\downarrow0} \mathcal{C}(\epsilon,\mu) = \mathcal{C}(\overline{\epsilon},0)$, and $\mathcal{C}(\epsilon_1,\mu_1) \supseteq \mathcal{C}(\epsilon_2,\mu_2)$ whenever $\epsilon_1 \geq \epsilon_2$ and $\mu_1\geq\mu_2$.
\end{proposition}

\begin{proof}
For any $(x,y,\lambda)\in\mathcal{C}(\epsilon_2,\mu_2)$, we have: $G(x) \leq 0$, $f(x,y) - h_{\mu_2}(\lambda,x) \leq \epsilon_2$, $g(x,y) \leq \epsilon_2$, and $\lambda \geq 0$.  Proposition \ref{prop:epicon} shows $[-h]_\mu(\lambda, x)$ is strictly decreasing in $\mu$, and so $f(x,y) - h_{\mu_1}(\lambda,x) \leq f(x,y) - h_{\mu_2}(\lambda,x) \leq\epsilon_2\leq\epsilon_1$.  Similarly, $g(x,y) \leq \epsilon_1\leq\epsilon_2$.  This shows $(x,y,\lambda)\in\mathcal{C}(\epsilon_1,\mu_1)$, which proves $\mathcal{C}(\epsilon_1,\mu_1) \supseteq \mathcal{C}(\epsilon_2,\mu_2)$ whenever $\epsilon_1 \geq \epsilon_2$ and $\mu_1\geq\mu_2$.  But $\mathcal{C}(\epsilon,\mu)$ is closed since $f,g,G$ are differentiable by \textbf{A1},\textbf{A3}; and $h_\mu$ is continuous by Proposition \ref{prop:hmulc2}. So the result follows by Exercise 4.3.b of \cite{rockafellar2009}.
\end{proof}

\begin{corollary}
Suppose $\mathbf{A1}$--$\mathbf{A3}$ and $\mathbf{R1}$,$\mathbf{R2}$ hold. If we have that $\epsilon\downarrow\overline{\epsilon},\mu\downarrow0,z\downarrow 0$, then
\begin{multline}
\limsup_{\epsilon\downarrow\overline{\epsilon},\mu\downarrow0,z\downarrow 0} \big(z\text{-}\arg\min \mathsf{DBP}(\epsilon,\mu)\big) \subseteq \\\vspace{-0.5cm}\arg\min \mathsf{DBP}(\overline{\epsilon},0),
\end{multline}
and $z\text{-}\min \mathsf{DBP}(\epsilon,\mu) \rightarrow \min \mathsf{DBP}(\overline{\epsilon},0)$.
\end{corollary}

\begin{proof}
Note $\mathsf{DBP}(\epsilon,\mu)$ is $f'_{\epsilon,\mu}(x,y,\lambda) = F(x,y) + \delta_{\mathcal{C}(\epsilon,\mu)}(x,y,\lambda)$.  Recall $\mathcal{C}(\epsilon,\mu)$ is closed since: $f,g,G$ are differentiable by \textbf{A1},\textbf{A3}; and $h_\mu$ is continuous by Proposition \ref{prop:hmulc2}.  By Proposition \ref{prop:convcons} we have $\mathcal{C}(\epsilon_1,\mu_1) \supseteq \mathcal{C}(\epsilon_2,\mu_2)$ when $\epsilon_1 \geq \epsilon_2$ and $\mu_1\geq\mu_2$, and so $f'_{\epsilon_1,\mu_1} \leq f'_{\epsilon_2,\mu_2}$ for $\epsilon_1 \geq \epsilon_2$ and $\mu_1\geq\mu_2$.  This means by Proposition 7.4.d of \cite{rockafellar2009} that $\mathsf{DBP}(\epsilon,\mu)$ epi-converges to $\mathsf{DBP}(\overline{\epsilon},0)$ as $\epsilon\downarrow\overline{\epsilon},\mu\downarrow0$.  Moreover, $\mathcal{C}(0,0)$ is nonempty by \textbf{R1} and Proposition \ref{prop:epssol}, which implies $\mathcal{C}(\epsilon,\mu)$ is nonempty.  So $f'_{\epsilon,\mu}(x,y,\lambda)$ is lower semicontinuous and feasible.  But $f'_{\epsilon,\mu}(x,y,\lambda) \geq f'_{\overline{\epsilon}+1,1}(x,y,\lambda)$ for $\epsilon \leq \overline{\epsilon}+1$ and $\mu \leq 1$ by Proposition \ref{prop:convcons}.  Moreover, $f'_{\epsilon,\mu}(x,y,\lambda) = F(x,y) + \delta_{\{(x,y) : (x,y,\lambda)\in\mathcal{C}(\epsilon,\mu)\}}(x,y) + \delta_{\mathcal{C}(\epsilon,\mu)}(x,y,\lambda)$.  Note $\{x : (x,y,\lambda) \in \mathcal{C}(\epsilon,\mu)\}$ is bounded by \textbf{R2}, and $\phi_\epsilon(x) = \{y : g(x,y) \leq \epsilon\}$ is: continuous by \textbf{A1},\textbf{R1} and Example 5.10 of \cite{rockafellar2009}; bounded for each $x\in X$ by \textbf{A1},\textbf{A2},\textbf{R1} and Corollary 8.7.1 of \cite{rockafellar1970}; and convex for each $x \in X$ by \textbf{A1}.  Hence applying Lemma \ref{lemma:btsvm} implies $\{y : (x,y,\lambda) \in \mathcal{C}(\epsilon,\mu)\}$ is bounded.  So $\{(x,y) : (x,y,\lambda)\in\mathcal{C}(\epsilon,\mu)\}$ is bounded, which by Example 1.11 of \cite{rockafellar2009} implies $\widetilde{f}_{\epsilon,\mu}(x,y,\lambda)$ is level bounded (see Definition 1.8 in \cite{rockafellar2009}) for all $\epsilon \leq \overline{\epsilon}+1$ and $\mu \leq 1$.  The result follows by Theorem 7.33 of \cite{rockafellar2009}.
\end{proof}

\section{Numerical Algorithm and Examples}

\label{section:numerics}


Previous sections provide theoretical justification for our Algorithm \ref{alg:dbpblp}, which uses \textsf{DBP} to solve \textsf{BLP}.  We conclude with two examples that demonstrate its effectiveness in solving practical problems.  The SNOPT solver \cite{gill2005snopt} was used for numerical optimization.  The first is a problem of inverse optimization with noisy data \cite{aswani2015,bertsimas2013,keshavarz2011}, and the second involves computing a Stackelberg strategy for routing games \cite{aswani2011,bonifaci2010,krichene2014,sharma2007,swamy2012}.

\begin{algorithm}[t]
\begin{algorithmic}[1] 
\caption{\textsf{DBP} Based Algorithm for Solving \textsf{BLP}}\label{alg:dbpblp}
\Require $x_0 \in X$; $\epsilon_0 > 0$; $\mu_0 \geq 0$; $\gamma,\zeta \in (0,1)$; $K \in \mathbb{Z}_+$
\For {$k = 1,\ldots,K-1$}
\State solve $\min_y \{f(x_k,y)\ |\ g(x_k,y) \leq 0\}$ using a convex optimization algorithm that provides a primal solution $y_k$ and the corresponding dual solution $\lambda_k$ 
\State solve \textsf{DBP}$(\epsilon_k, \mu_k)$ using a nonlinear optimization algorithm with the derivatives of \textsf{RDF} as in (\ref{eqn:defnab}), and with initial (feasible) point $(x_k, y_k, \lambda_k)$; set $x_{k+1}$ to be the computed minimizer in the $x$ variable
\State set $(\epsilon_{k+1},\mu_{k+1}) \leftarrow (\gamma\cdot\epsilon_k,\zeta\cdot\mu_k)$
\EndFor
\State return $x_K$
\end{algorithmic}
\end{algorithm}

\subsection{Inverse Optimization with Noisy Data}

Suppose an agent decides $y_i$ in response to a signal $u_i$ by maximizing a utility function $U(y, u, \theta_0)$, where $\theta_0$ is a vector of parameters.  Statistically consistent estimation of $\theta_0$ given $(u_i, z_i)$ for $i=1,\ldots,n$ data points, where $z_i$ are noisy measurements of $y_i$, requires solving \textsf{BLP} \cite{aswani2015}.  Heuristics using convex optimization (like \cite{bertsimas2013,keshavarz2011}) are inconsistent \cite{aswani2015}.

If $U(y, u, x) = -(x + u)y$ with $x,y,u\in\mathbb{R}$, then the bilevel program for statistical estimation is 
\begin{equation}
\label{eqn:blpest}
\begin{aligned}
\min_{x,y_i}\ & \textstyle\frac{1}{n}\sum_{i=1}^n\|z_i - y_i\|^2\\
\text{s.t. } & x \in [-1,1]\\
& y_i \in\arg\min_y \{(x+u_i)y\ |\ y \in [-1,1]\}, \forall i = 1,\ldots,n
\end{aligned}
\end{equation}
The reformulation \textsf{DBP}$(\epsilon,\mu)$ for this instance is given by
\begin{equation}
\begin{aligned}
\min_{x,y_i}\ & \textstyle\frac{1}{n}\sum_{i=1}^n\|z_i - y_i\|^2\\
\text{s.t. } & x \in [-1,1] \\
&(x+u_i)y_i - h_\mu(\lambda_i, x) - \epsilon \leq 0,\ \forall i = 1,\ldots,n\\
&y_i \in [-1-\epsilon,1+\epsilon],\ \lambda_i \geq 0,\ \forall i = 1,\ldots,n
\end{aligned}
\end{equation}
where $\lambda_i\in\mathbb{R}^2$, and the \textsf{RDF} is $\textstyle h_\mu(\lambda_i,x) = \min_y\{\mu\cdot y^2 + (x+u_i)\cdot y + \lambda_{i,1}\cdot(-y-1) + \lambda_{i,2}\cdot(y-1)\ |\ y\in[-2,2]\}$.  

Two hundred instances of (\ref{eqn:blpest}) with $n = 100$ were solved, where (a) $u_i$ and $\theta_0$ were drawn from a uniform distribution over $[-1,1]$, and (b) $z_i = \xi_i + w_i$ with $\xi_i \in \arg\min_y \{(\theta_0+u_i)y\ |\ y \in [-1,1]\}$ and $w_i$ drawn from a standard normal.  Each instance was solved by Algorithm \ref{alg:dbpblp}, where: $x_0$ was drawn from a uniform distribution over $[-1,1]$, $\epsilon_0 = 1$, $\mu_0 = 10^{-4}$, $\gamma=0.1$, $\zeta=1$, and $K=3$.  We use $\hat{\theta}$ to refer to the value returned by Algorithm \ref{alg:dbpblp}, to emphasize that the returned value is an estimate of $\theta_0$.  Fig. \ref{fig:iopt} has scatter plots of the 200 solved instances; it shows (left) the initial (randomly chosen) $x_0$ are uncorrelated to the true $\theta_0$, and (right) the estimates $\hat{\theta}$ computed using our algorithm are close to the true $\theta_0$.

\begin{figure}
\centering
 \subfloat[Initialized]{\includegraphics[trim=0.17in 0 0.23in 0, clip]{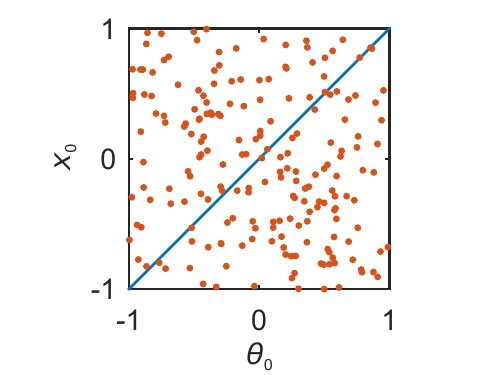}}\hspace{0cm}
 \subfloat[Estimated]{\includegraphics[trim=0.17in 0 0.23in 0, clip]{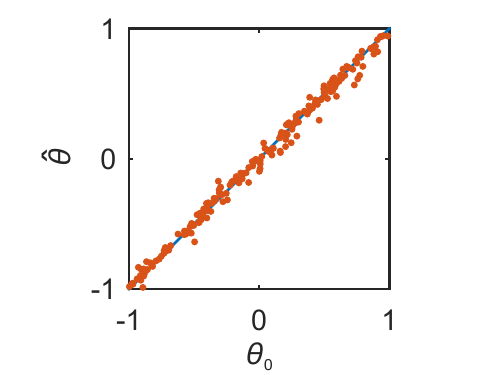}}
\caption{Scatter plot of initialized $x_0$ versus true parameter $\theta_0$ (Left).  Scatter plot of estimated (by Algorithm \ref{alg:dbpblp}) $\hat{\theta}$ versus true parameters $\theta_0$ (Right).}
\label{fig:iopt}
\end{figure}

\subsection{Stackelberg Routing Games}

A common class of routing games consists of a directed graph with multiple edges between vertices, convex delay functions for each edge, and a listing of inflows and outflows of traffic \cite{aswani2011,bonifaci2010,krichene2014,sharma2007,swamy2012}.  The Stackelberg strategy is a situation where a leader controls an $\alpha$ fraction of the flow, the remaining flow is routed according to a Nash equilibrium given the flow of the leader, and the leader routes their flow to minimize the average delay in the network.  This problem is a bilevel program with a convex lower level.

An example of a two edge network in this Stackelberg setting is shown below:\\

{\centering
\includegraphics[width=2in,trim=3.8in 2in 3.8in 2.5in, clip]{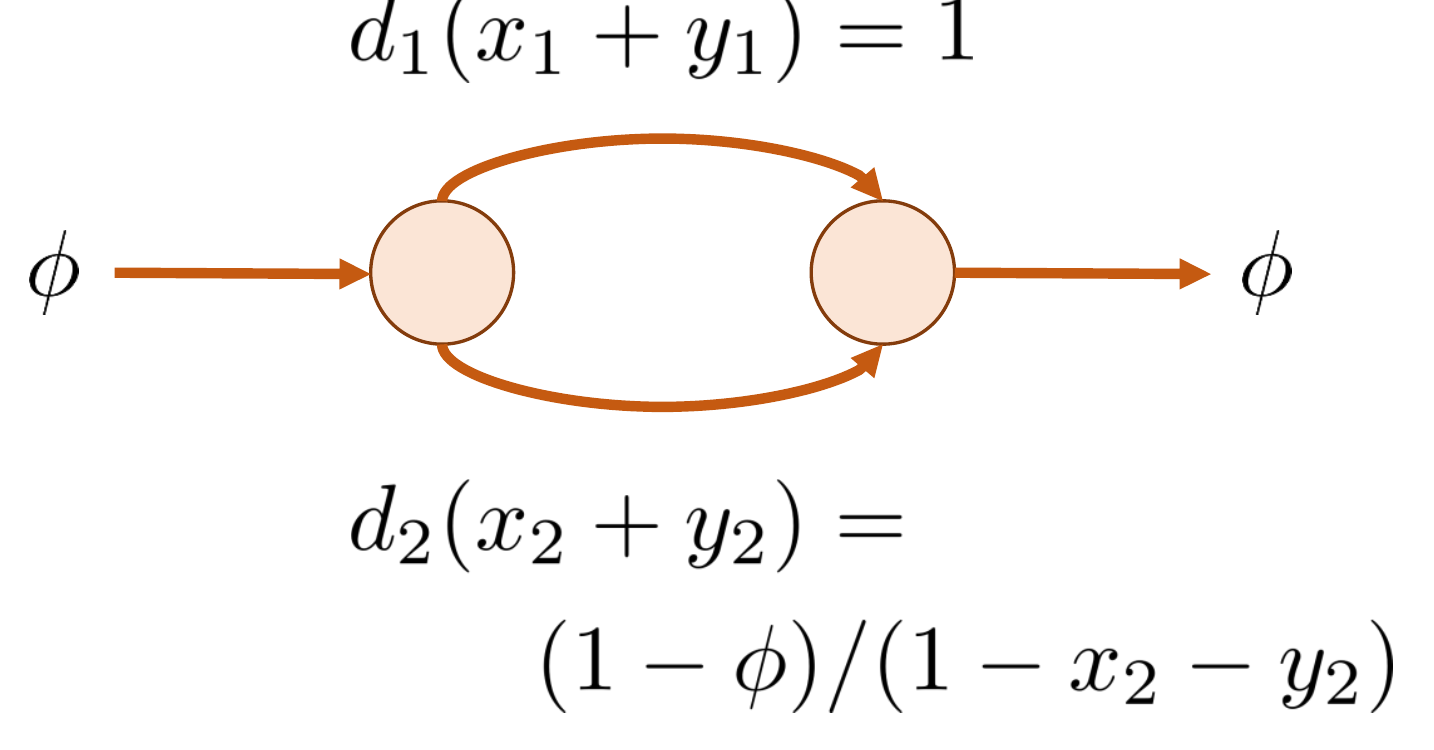}\\
}
\noindent The Stackelberg strategy for this two edge network is the solution to 
\begin{equation}
\label{eqn:ssblp}
\begin{aligned}
\min_{x,y}\ & x_1 + y_1 + (1-\phi)\cdot(x_2 + y_2)/(1 - x_2 - y_2)\\
\text{s.t. } & x_1 + x_2 = \alpha\cdot\phi,\quad x\geq 0\\
&y \in \arg\min_y \big\{x_1 + y_1 - (1-\phi)\cdot\log(1-x_2-y_2)\\
&\hspace{2cm}\ \big|\ y_1+y_2=(1-\alpha)\cdot\phi,\ y\geq 0\big\}
\end{aligned}
\end{equation}
where (a) $x_1,x_2$ is the leader's flow on the top/bottom edge, (b) $y_1,y_2$ is the follower's flow on the top/bottom edge, (c) $\phi < 1$ is the amount of flow entering the network, and (d) $\alpha$ is the fraction of the flow controlled by the leader.  The duality-based reformulation \textsf{DBP}$(\epsilon,\mu)$ for this instance is
\begin{equation}
\begin{aligned}
\min_{x,y}\ & x_1 + y_1 + (1-\phi)\cdot(x_2 + y_2)/(1 - x_2 - y_2)\\
\text{s.t. } & x_1 + x_2 = \alpha\cdot\phi,\quad x,y,\lambda \geq 0 \\
& x_1 + y_1 - (1-\phi)\cdot\log(1-x_2-y_2) +\\
&\hspace{3.5cm}- h_\mu(\lambda,\nu,x) - \epsilon \leq 0\\
& y_1+y_2 \in [(1-\alpha)\cdot\phi - \epsilon, (1-\alpha)\cdot\phi + \epsilon]
\end{aligned}
\end{equation}
where $\lambda\in\mathbb{R}^2$, $\nu\in\mathbb{R}$, and the \textsf{RDF} is $h_\mu(\lambda,\nu,x) = \min_y\{\mu\cdot \|y\|^2 + x_1 + y_1 - (1-\phi)\cdot\log(1-x_2-y_2) -\lambda_1\cdot y_1 - \lambda_2\cdot y_2 + \nu\cdot(y_1+y_2-(1-\alpha)\cdot\phi)\ |\ y\in[-1,2]\}$.  Different instances (corresponding to different values of $\alpha,\phi$) were solved by Algorithm \ref{alg:dbpblp}, where: $\epsilon_0 = 1$, $\mu_0 = 10^{-4}$, $\gamma=0.1$, $\zeta=1$, and $K=3$.  The initial point provided to the algorithm was the \textsf{SCALE} strategy \cite{aswani2011,bonifaci2010,swamy2012}, which corresponds to computing $x'\in\arg\min_x \{x_1 + (1-\phi)\cdot(x_2)/(1 - x_2)\ |\ x_1 + x_2 = \phi,\ x\geq 0\}$ and then choosing $\alpha x'$ as the initial point.

Solution quality is evaluated by the price of anarchy (\textsf{PoA}) \cite{roughgarden2003}, which is the average delay of a solution divided by the average delay when $\alpha = 1$.  The objective in (\ref{eqn:ssblp}) gives the average delay.  A \textsf{PoA} close to 1 is ideal because it implies the delay of the strategy is close to the delay when the leader controls the entire flow, while a large \textsf{PoA} means the average delay of the strategy is much higher than when the leader controls the entire flow.  The results in Fig. \ref{fig:stack} show that our duality-based approach (initialized with \textsf{SCALE}) significantly improves the quality of the Stackelberg strategy.

\begin{figure}
\centering
 \subfloat[\textsf{SCALE} Strategy]{\includegraphics[scale=0.85]{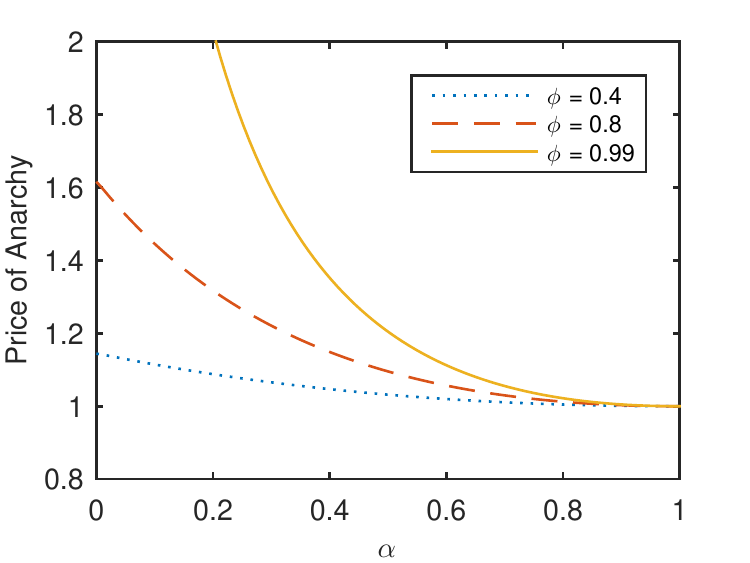}}\\
 \subfloat[Duality-Based Strategy]{\includegraphics[scale=0.85]{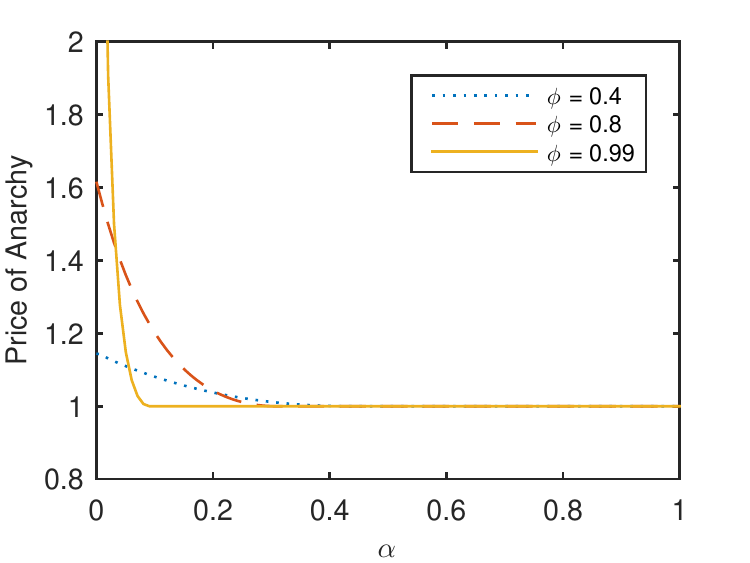}}
\caption{Comparison of Stackelberg strategy quality for $\mathsf{SCALE}$ strategy (Top), which was used as an initialization to compute the duality-based strategy (Bottom) using our Algorithm \ref{alg:dbpblp}.}
\label{fig:stack}
\end{figure}

\section{Conclusion}

We used a new (differentiable) dual function to construct a duality-based reformulation of bilevel programs with a convex lower level, and this reformulation uses regularization to ensure constraint qualification and differentiability.  We proved results about the properties of this reformulation as justification for a new algorithm to solve bilevel programs, and then we displayed the effectiveness of our algorithm by solving two practical instances of bilevel programming.






\bibliographystyle{IEEEtran}
\bibliography{IEEEabrv,dual}

\begin{thebibliography}{10}
\providecommand{\url}[1]{#1}
\csname url@samestyle\endcsname
\providecommand{\newblock}{\relax}
\providecommand{\bibinfo}[2]{#2}
\providecommand{\BIBentrySTDinterwordspacing}{\spaceskip=0pt\relax}
\providecommand{\BIBentryALTinterwordstretchfactor}{4}
\providecommand{\BIBentryALTinterwordspacing}{\spaceskip=\fontdimen2\font plus
\BIBentryALTinterwordstretchfactor\fontdimen3\font minus
  \fontdimen4\font\relax}
\providecommand{\BIBforeignlanguage}[2]{{%
\expandafter\ifx\csname l@#1\endcsname\relax
\typeout{** WARNING: IEEEtran.bst: No hyphenation pattern has been}%
\typeout{** loaded for the language `#1'. Using the pattern for}%
\typeout{** the default language instead.}%
\else
\language=\csname l@#1\endcsname
\fi
#2}}
\providecommand{\BIBdecl}{\relax}
\BIBdecl

\bibitem{aswani2011}
A.~Aswani and C.~Tomlin, ``Game-theoretic routing of {GPS}-assisted vehicles
  for energy efficiency,'' in \emph{ACC}, 2011, pp. 3375--3380.

\bibitem{bonifaci2010}
V.~Bonifaci, T.~Harks, and G.~Sch{\"a}fer, ``Stackelberg routing in arbitrary
  networks,'' \emph{Math. Oper. Res.}, vol.~35, no.~2, pp. 330--346, 2010.

\bibitem{krichene2014}
W.~Krichene, J.~D. Reilly, S.~Amin, and A.~M. Bayen, ``Stackelberg routing on
  parallel networks with horizontal queues,'' \emph{{IEEE} Trans. Automat.
  Contr.}, vol.~59, no.~3, pp. 714--727, 2014.

\bibitem{sharma2007}
Y.~Sharma and D.~P. Williamson, ``Stackelberg thresholds in network routing
  games or the value of altruism,'' in \emph{ACM conference on Electronic
  commerce}, 2007, pp. 93--102.

\bibitem{swamy2012}
C.~Swamy, ``The effectiveness of {Stackelberg} strategies and tolls for network
  congestion games,'' \emph{ACM TALG}, vol.~8, no.~4, p.~36, 2012.

\bibitem{aswani2015}
A.~Aswani, Z.-J.~M. Shen, and A.~Siddiq, ``Inverse optimization with noisy
  data,'' \emph{arXiv:1507.03266}, 2015.

\bibitem{bertsimas2013}
D.~Bertsimas, V.~Gupta, and I.~C. Paschalidis, ``Data-driven estimation in
  equilibrium using inverse optimization,'' \emph{Math Prog}, 2014.

\bibitem{keshavarz2011}
A.~Keshavarz, Y.~Wang, and S.~Boyd, ``Imputing a convex objective function,''
  in \emph{IEEE ISIC}, 2011, pp. 613--619.

\bibitem{ng2000}
A.~Ng and S.~Russell, ``Algorithms for inverse reinforcement learning.'' in
  \emph{ICML}, 2000, pp. 663--670.

\bibitem{vasudevan2012safe}
R.~Vasudevan, V.~Shia, Y.~Gao, R.~Cervera-Navarro, R.~Bajcsy, and F.~Borrelli,
  ``Safe semi-autonomous control with enhanced driver modeling,'' in
  \emph{ACC}, 2012, pp. 2896--2903.

\bibitem{sadigh2016planning}
D.~Sadigh, S.~Sastry, S.~Seshia, and A.~Dragan, ``Planning for autonomous cars
  that leverages effects on human actions,'' in \emph{RSS}, 2016.

\bibitem{mintz2017behavioral}
Y.~Mintz, A.~Aswani, P.~Kaminsky, E.~Flowers, and Y.~Fukuoka, ``Behavioral
  analytics for myopic agents,'' \emph{arXiv preprint arXiv:1702.05496}, 2017.

\bibitem{anitescu2005}
M.~Anitescu, ``On using the elastic mode in nonlinear programming approaches to
  mathematical programs with complementarity constraints,'' \emph{SIAM J
  Optim.}, vol.~15, no.~4, pp. 1203--1236, 2005.

\bibitem{fukushima1999}
M.~Fukushima and J.-S. Pang, ``Convergence of a smoothing continuation method
  for mathematical programs with complementarity constraints,'' in
  \emph{Ill-posed Variational Problems and Regularization Techniques}.\hskip
  1em plus 0.5em minus 0.4em\relax Springer, 1999, pp. 99--110.

\bibitem{miguel2004}
A.~V.~d. Miguel, M.~P. Friedlander, F.~J. Nogales~Mart{\'\i}n, and S.~Scholtes,
  ``An interior-point method for {MPECs} based on strictly feasible
  relaxations.'' Department of Decision Sciences, London Business School, Tech.
  Rep., 2004.

\bibitem{outrata1990}
J.~V. Outrata, ``On the numerical solution of a class of {Stackelberg}
  problems,'' \emph{Zeitschrift f{\"u}r Operations Research}, vol.~34, no.~4,
  pp. 255--277, 1990.

\bibitem{ye1995}
J.~Ye and D.~Zhu, ``Optimality conditions for bilevel programming problems,''
  \emph{Optimization}, vol.~33, no.~1, pp. 9--27, 1995.

\bibitem{kanzow2014}
C.~Kanzow and A.~Schwartz, ``The price of inexactness: convergence properties
  of relaxation methods for mathematical programs with complementarity
  constraints revisited,'' \emph{Mathematics of Operations Research}, vol.~40,
  no.~2, pp. 253--275, 2014.

\bibitem{lin2005}
G.-H. Lin and M.~Fukushima, ``A modified relaxation scheme for mathematical
  programs with complementarity constraints,'' \emph{Annals of Operations
  Research}, vol. 133, no. 1-4, pp. 63--84, 2005.

\bibitem{scholtes2001}
S.~Scholtes, ``Convergence properties of a regularization scheme for
  mathematical programs with complementarity constraints,'' \emph{SIAM Journal
  on Optimization}, vol.~11, no.~4, pp. 918--936, 2001.

\bibitem{steffensen2010}
S.~Steffensen and M.~Ulbrich, ``A new relaxation scheme for mathematical
  programs with equilibrium constraints,'' \emph{SIAM Journal on Optimization},
  vol.~20, no.~5, pp. 2504--2539, 2010.

\bibitem{lin2014}
G.-H. Lin, M.~Xu, and J.~Ye, ``On solving simple bilevel programs with a
  nonconvex lower level program,'' \emph{Mathematical Programming}, vol. 144,
  no. 1-2, pp. 277--305, 2014.

\bibitem{rockafellar2009}
R.~T. Rockafellar and R.~J.-B. Wets, \emph{Variational analysis}, 3rd~ed.\hskip
  1em plus 0.5em minus 0.4em\relax Springer, 2009.

\bibitem{hiriart1979}
J.-B. Hiriart-Urruty, ``Refinements of necessary optimality conditions in
  nondifferentiable programming {I},'' \emph{Applied mathematics and
  optimization}, vol.~5, no.~1, pp. 63--82, 1979.

\bibitem{bonnans2000}
J.~Bonnans and A.~Shapiro, \emph{Perturbation Analysis of Optimization
  Problems}.\hskip 1em plus 0.5em minus 0.4em\relax Springer, 2000.

\bibitem{rockafellar1982}
R.~Rockafellar, ``Favorable classes of {Lipschitz} continuous functions in
  subgradient optimization,'' in \emph{Progress in Nondifferentiable
  Optimization}.\hskip 1em plus 0.5em minus 0.4em\relax IIASA, 1982, pp.
  125--143.

\bibitem{clarke1975}
F.~H. Clarke, ``Generalized gradients and applications,'' \emph{Transactions of
  the American Mathematical Society}, vol. 205, pp. 247--262, 1975.

\bibitem{berge1963}
C.~Berge, \emph{Topological Spaces: including a treatment of multi-valued
  functions, vector spaces, and convexity}.\hskip 1em plus 0.5em minus
  0.4em\relax Dover, 1963.

\bibitem{nemirovski2004}
A.~Nemirovski, ``Interior point polynomial time methods in convex
  programming,'' Georgia Institute of Technology, Tech. Rep., 2004.

\bibitem{nesterov1994}
Y.~Nesterov and A.~Nemirovskii, \emph{Interior-Point Polynomial Algorithms in
  Convex Programming}.\hskip 1em plus 0.5em minus 0.4em\relax SIAM, 1994.

\bibitem{dempe2012}
S.~Dempe and J.~Dutta, ``Is bilevel programming a special case of a
  mathematical program with complementarity constraints?'' \emph{Mathematical
  programming}, vol. 131, no. 1-2, pp. 37--48, 2012.

\bibitem{wachsmuth2013}
G.~Wachsmuth, ``On {LICQ} and the uniqueness of {Lagrange} multipliers,''
  \emph{Operations Research Letters}, vol.~41, no.~1, pp. 78--80, 2013.

\bibitem{scheel2000}
H.~Scheel and S.~Scholtes, ``Mathematical programs with complementarity
  constraints: Stationarity, optimality, and sensitivity,'' \emph{Math. Oper.
  Res.}, vol.~25, no.~1, pp. 1--22, 2000.

\bibitem{polak1997}
E.~Polak, \emph{Optimization: algorithms and consistent approximations}.\hskip
  1em plus 0.5em minus 0.4em\relax Springer Science \& Business Media, 1997,
  vol. 124.

\bibitem{royset2016}
J.~O. Royset and R.~J. Wets, ``Optimality functions and lopsided convergence,''
  \emph{J Optim Theory Appl}, pp. 1--19, 2016.

\bibitem{rockafellar1970}
R.~T. Rockafellar, \emph{Convex Analysis}.\hskip 1em plus 0.5em minus
  0.4em\relax Princeton University Press, 1970.

\bibitem{gill2005snopt}
P.~E. Gill, W.~Murray, and M.~A. Saunders, ``{SNOPT},'' \emph{SIAM review},
  vol.~47, no.~1, pp. 99--131, 2005.

\bibitem{roughgarden2003}
T.~Roughgarden, ``The price of anarchy is independent of the network
  topology,'' \emph{J. Comput. Syst. Sci.}, vol.~67, no.~2, pp. 341--364, 2003.

\end{thebibliography}

\end{document}